\documentclass{amsart}

\newcommand{\cI}{\mathcal I}
\newcommand{\cM}{\mathcal M}

\newcommand{\higho}{$^{\text{o}}$}

\newtheorem{thm}{Theorem}

\newtheorem{lem}{Lemma}

\author{Franz Lemmermeyer}
\title{Euler, Goldbach, and ``Fermat's Theorem''}
\begin{document}

\begin{abstract}
While preparing the correspondence between Leonhard Euler and
Christian Goldbach for publication, Martin Mattm\"uller asked
whether the lemma given in the postscript of Euler's letter 
dated July 26, 1749, was enough for completing the proof of 
Fermat's Four Squares Theorem. In this article we will show 
that Euler's result can in fact be used for proving this result 
via induction.
\end{abstract}

\maketitle
\markboth{Euler, Goldbach, and Sums of Squares}
         {\today \hfil Franz Lemmermeyer}

\section*{Introduction}

The correspondence between Leonhard Euler and Christian Goldbach is 
a rich source for studying the development of Euler's work in number 
theory. It was first published by P.H. Fu\ss{} \cite{Fuss} in 1843, 
and then again by A.P. Jushkevich and E. Winter \cite{EG65} in 1965. 
The correspondence, both in the original mixture of Latin and German, 
as well as in an English translation, is scheduled to appear as vol. 4 of 
Series IV-A of Euler's Opera Omnia \cite{EuGo} at the end of 2011.

Many letters between Euler and Goldbach deal with various 
number theoretic problems first posed (and sometimes solved) by 
Pierre Fermat. Here we discuss his results on sums of two and 
four squares. As early as September 1636, Fermat stated the 
Polygonal-Number Theorem in a letter to Mersenne: every positive 
integer is the sum of (at most) three triangular numbers, four 
squares, five pentagonal numbers etc.:
\begin{quote}
1. Every number is the sum of one, two or three triangular numbers, 
\begin{align*}
& \text{one}, 2, 3, 4           & \ldots & & \text{squares}, \\
& \text{one}, 2, 3, 4, 5        & \ldots & & \text{pentagonal numbers}, \\
& \text{one}, 2, 3, 4, 5, 6     & \ldots & & \text{hexagonal numbers}, \\
& \text{one}, 2, 3, 4, 5, 6, 7  & \ldots & & \text{heptagonal numbers}, 
\end{align*}
and so on until infinity.

It seems that Diophantus\footnote{The passage in Diophantus which 
Fermat is referring to is problem 31 in Book IV; for Fermat's 
comments on this problem, see \cite[I, p. 305]{Ferm}. In Heath's 
edition \cite[p. 188]{Heath}, this is problem 29 in Book IV.} 
assumed the second part of the theorem, and 
Bachet tried to verify it empirically, but did not attain a demonstration.
\end{quote}
Fermat then continues
\begin{quote}
2. The eightfold multiple of an arbitrary
number, diminished by $1$, is composed of four squares -- not only
in integers -- which perhaps others might have already seen --  but 
also in fractions, as I promise to prove.
\end{quote}
The point Fermat is trying to make is that primes of the form
$8n-1$ cannot be written as a sum of less than four {\em rational} 
squares. A brief summary of the most important letters concerning
sums of squares is given in the following table:

\begin{table}[h!]
\begin{tabular}{|p{1.7cm}|p{1.8cm}|p{8cm}|} \hline 
   date & written to & content \\ \hline 
  15.07.1636 & Mersenne & A number $n$ is a sum of 
                           exactly three integral squares
                           if and only if $a^2n$ is \\
  02.09.1636 & Mersenne & A number is a sum of three integral
                          squares if and only if it is a sum
                          of three rational squares. \\
  16.09.1636 & Roberval & If $a$ and $b$ are rational,
                          and if $a^2 + b^2 = 2(a+b)x+x^2$,
                          then $x$ and $x^2$ are irrational. \\
  Sept. 1636 & Mersenne & F. asks for solutions of $x^4 + y^4 = z^4$ and 
                          $x^3 + y^3 = z^3$, and states
                          the polygonal number theorem. He claims that 
                          every integer $8n-1$ is the sum of four 
                          squares, but not of three; both in 
                          integers and fractions. \\  
  May 1640   & Mersenne & Fermat repeats the problems he
                          communicated in Sept. 1636 \\
  Dec. 1640  & Mersenne & Fermat states Two-Squares Theorem \\
  June 1658  & Digby    & Fermat claims proof of the 
                          Two Squares Theorem. \\
  Aug. 1659  & Carcavi  & Fermat claims proof of the 
                          Four Squares Theorem. \\
\hline
\end{tabular} \smallskip
\end{table} 
In addition we remark that in a letter to Descartes dated March 22, 1638, 
Mersenne reports that Fermat is able to prove that no number of the form 
$4n-1$ is a sum of two integral or rational squares.

\section{The Four-Squares Theorem in the Euler-Goldbach Correspondence}

\begin{table}[h!]
\begin{tabular}{|r|p{2.5cm}|p{8.5cm}|} \hline 
   \# & letter & content \\ \hline 
  2 & Dec. 1, 1729 & Goldbach asks whether Euler knows Fermat's
                     claim that all numbers $2^{2^n}+1$ are prime. \\
  3 & Jan. 8, 1730 & Euler is unable to do anything with 
                     Fermat's problem. \\
  4 & May 22, 1730 & Goldbach explains how to compute with remainders. \\
  5 & June 4, 1730 & Euler observes that $2^{n}+1$ is composite of
                     $n$ has an odd prime divisor. 

                     ``Lately, reading Fermat's works, I came upon
                    another rather elegant theorem stating that any 
                    number is the sum of four squares, or that for 
                    any number four square numbers can be found whose 
                    sum is equal to the given number''.  \\
  6 & June 26, 1730 & Goldbach has not read Fermat's works. \\
  7 & June 25, 1730 & Euler observes that $10^4+1$ is divisible by $37$,
                     and that $3^8+2^8$ is divisible by $17$. Euler 
                     cannot prove that any number is the sum of four 
                     squares. He has found another result by Fermat,
                     namely that $1$ is the only triangular number that
                     is a fourth power (Several years earlier, Goldbach 
                     had sent an erroneous proof of this claim to 
                     D. Bernoulli). \\
  8 & July 31, 1730 & Goldbach proves that Fermat numbers are pairwise
                     coprime. He claims that $1$ is the only square 
                     among the triangular numbers. \\
  9 & Aug. 10, 1730 & Euler mentions that Fermat and Wallis studied
                     the equation $ap^2 + 1 = q^2$, and mentions a 
                     method for solving it which he credits to Pell. \\
 10 & Oct.  9, 1730 & Goldbach studies sums of three and four squares. \\
 11 & Oct. 17, 1730 & Euler mentions another theorem by Fermat:
                     ``any number is the sum of three triangular 
                       numbers''. \\
 15 & Nov. 25, 1730 & By studying prime divisors of numbers $2^p-1$,
                      Euler discovered ``Fermat's Little Theorem''. \\
 40 & Sept. 9, 1741 & Euler studies prime divisors of $x^2 + y^2$,
                       $x^2 - 2y^2$, and $x^2 - 3y^2$.  \\
 47 & March 6, 1742 & Euler proves ``a theorem of Fermat's'' according
                   to which primes $p = 4n+3$ cannot divide a sum of two 
                   squares $a^2 + b^2$ except when both $a$ and $b$ are 
                   divisible by $p$. \\
 52 & June 30, 1742 &  Euler claims that prime numbers $4n+1$ are
                   represented uniquely as a sum of two squares. 
                   He also mentions that $641$ divides $2^{32}+1$,
                   thereby disproving Fermat's claim that all numbers
                   $2^{2^n}+1$ are prime.  \\
 56 & Oct. 27, 1742 &  Euler has written to Clairaut, asking 
                   him ``whether Fermat's manuscripts might still 
                   be found''. \\
 72 & Aug. 24, 1743 &  Euler sketches the idea of infinite descent. \\
 73 & Sept. 28, 1743 & Goldbach, with considerable help by Euler, 
                   gives a new proof of Euler's result that primes 
                   $p = 4n+3$ do not divide numbers of the form $a^2+1$. \\
\hline
\end{tabular} 
\end{table} 

\begin{table}[h!]
\begin{tabular}{|r|p{2.5cm}|p{8.5cm}|} \hline 
   \# & letter & content \\ \hline 
 74 & Oct. 15, 1743 &  Euler claims that if a number is a sum of 
                    two (three, four) rational squares, then it is a sum of 
                    two (three, four) integral squares. \\
 87 & Feb. 16, 1745 & Euler shows that numbers represented in  
                    two different ways as a sum of two squares must 
                    be composite. \\ 
114 & April 15, 1747 & Goldbach is skeptical about some of Fermat's 
                   claims, i.e. that every number is a sum of three 
                   triangular numbers, or that every integer $8n+3$ 
                   is the sum of three squares. \\
115 & May 6, 1747 &  Euler proves the Two-Squares Theorem except for
                   the following lemma: there exist integers $a$, $b$ 
                   such that $a^n - b^n$ is not 
                   divisible by the prime $4n+1$.  \\
125 & Feb. 13, 1748 & Euler writes that the proof of the Three-Squares
                   theorem ought to resemble his proof for two squares. 
                   Euler mentions ``Fermat's Last Theorem''. \\
126 & April 6, 1748 & Goldbach observes that if $2n+1$ is a sum of 
                   three squares, then $2n+3$, $4n+3$, $4n+6$ and
                   $6n+3$ are sums of four squares. \\
127 & May 4, 1748 & Euler states the product formula for sums of 
                   four squares. He also suggests proving theorems such as 
                   the Four-Squares Theorem using generating functions. \\
138 & April 12, 1749 & Euler closes the gap in his proof n\higho\, 115. 
                    He can prove the Four-Squares Theorem except for the 
                    lemma: If $ab$ and $b$ are sums of four 
                    squares, then so is $a$. \\
140 & July 16, 1749 & Goldbach knows how to prove the following special 
                    case of Euler's missing lemma: if $8m+4$ is a sum 
                    of four odd squares, then $2m+1$ is a sum of four 
                    squares. \\ 
141 & July 26, 1749 & Euler observes that the Four-Squares Theorem
                    follows if it can be shown to hold for all
                    numbers of the form $n = 8n+1$ (or, more generally,
                    for all numbers of any of the forms $8n + a$
                    with $a = 1$, $3$, $5$ or $7$.  

                    Euler also proves special cases of the 
                    ``missing link'' in his proof of the Four-Squares 
                    Theorem: if $pA$ is a sum of four squares
                    and $p = 2$, $3$, $5$, $7$, then so is $A$. He also 
                    formulates a general lemma that brings him within 
                    inches of a full proof. \\
144 & June 9, 1750 & Euler laments the fact that he can prove that 
                    every natural number is the sum of four rational 
                    squares, but that he cannot do it for integers. \\
147 & Aug. 17, 1750 & Euler returns to his idea of using generating
                    functions for proving the Four-Squares Theorem. \\
169 & Aug. 4, 1753 & Euler mentions ``another very beautiful theorem''
                    in Fermat's work: ``Fermat's Last Theorem''. He 
                    remarks that he has found a proof for exponent $3$. \\
 \hline
\end{tabular} \smallskip
\caption{Fermat's Theorems in the Euler-Goldbach Correspondence}
\end{table}

In this article we describe Euler's efforts at proving the
Four-Squares Theorem. As we will see, using the lemma which
Euler ``almost'' proved in his letter no. 141 it is an easy
exercise to complete the proof. In order to see how natural
Euler's approach is, we will first discuss a proof of the 
Two-Squares Theorem based on the same principles. The first 
published proof of the Four-Squares Theorem is due to 
Lagrange \cite{Lagr}; immediately afterwards, Euler \cite{E445}
simplified Lagrange's version.

There are perhaps no better examples in Goldbach's correspondence 
with Euler  for illuminating his role as a catalyst than the letters 
discussing various aspects of the Four-Squares Theorem.

In his letter [EG126; April 6, 1748] to Euler, Goldbach 
writes\footnote{The excerpts from the correspondence Euler--Goldbach 
are all taken from \cite{EuGo}; the translation into English is due to
Martin Mattm\"uller.}
\begin{quote}
If you can prove, as you think you can, that all numbers $8m+3$ can be
brought to the form $2a^{2}+b^{2}$ if they are prime you will also 
easily find that all prime numbers $4m+3$ belong to the  formula 
$2a^{2}+b^{2}+c^{2}$, since in my opinion this comprises all odd 
numbers; but if this were proved just for all prime numbers, it 
should be obvious that all positive numbers consist of four squares.
\end{quote}
Goldbach thus thought that once Euler could prove that every prime
$p = 8n+3$ has the form $p = 2a^2 + b^2$, he should also be able 
to prove\footnote{In his reply, Euler remarks that he is unable to 
deduce the second claim from the first: 
\begin{quote}
If the proposition that $8m+3$ equals $2a^{2}+b^{2}$ whenever $8m+3$ 
is a prime number is true, I do not see that $4n+3$ must always equal 
$2a^{2}+b^{2}+c^{2}$ whenever $4n+3$ is a prime number.
\end{quote}}  
the claim that every prime $4m+3$ can be written in the 
form $2a^2 + b^2 + c^2$. The claim that every odd number $2m+1$ is 
represented by the ternary quadratic form
$2a^2+b^2+c^2$ is equivalent 
to $4m+2 = 4a^2 + 2b^2 + 2c^2 = (2a)^2 + (b-c)^2 + (b+c)^2$,
hence follows from a special case of the Three-Squares Theorem.

Goldbach also observes that if $2n+1 = 2a^2+b^2+c^2$, then 
e.g. $3(2n+1) = 6n+3 = (a+b+c)^2 + (a+b-c)^2 + (2a-b)^2 + c^2$ is 
a sum of four squares. In his reply [EG127; May 4, 1748], Euler 
shows that Goldbach's observations are special cases of the following
product formula: if $m = a^2 + b^2 + c^2 + d^2$ and 
$n = x^2 + y^2 + z^2 + v^2$, then $mn = f^2 + g^2 + h^2 + k ^2$ 
for\footnote{Euler's notation and choice of signs differ from 
the formulas given here.}
      \begin{equation}\label{E4sq}
         \left\{\begin{array}{rclrcl}
              f & = & ax + by + cz + dv,  &  g & = & bx - ay - dz + cv, \\
              h & = & cx + dy - az - bv,  &  k & = & dx - cy + bz - av.
                 \end{array} \right.
       \end{equation}
Actually, Euler had known the formula at least since 1740, as his
notebooks (see Pieper \cite{Pieper}) show.

A year later, on April 12, 1749, Euler returns to the problem of 
Four Squares and remarks
\begin{quote}
I can almost prove that any number is a sum of four or fewer 
squares; indeed, what I am lacking is just one proposition, 
which does not appear to present any difficulty at first sight.
\end{quote}
In fact, Euler [EG138] announces a plan for proving the theorem:
he introduces the symbol $\fbox{$4$}$ for denoting sums of 
four (or fewer) squares, and then states:
\begin{enumerate}
\item If $a{}=\fbox{$4$}$ and $b=\fbox{$4$}$, then also $ab=\fbox{$4$}$.
\item If $ab=\fbox{$4$}$ and $a=\fbox{$4$}$, then also $b=\fbox{$4$}$. 
\item Corollary: \ldots if $ab=\fbox{$4$}$ and $a\neq \fbox{$4 $}$ \ldots, 
      then also $b\neq \fbox{$4$}$.
\item If all prime numbers were of the form $\fbox{$4$}$, then every 
      number at all should be contained in the form $\fbox{$4$}${}.
\item An arbitrary prime number $p$ being proposed, there always is 
      some number of the form $a^{2}+b^{2}+c^{2}+d^{2}$ which is 
      divisible by $p$, while none of the numbers $a$, $b$, $c$, $d$ 
      themselves is divisible by $p$.
\item If $a^{2}+b^{2}+c^{2}+d^{2}$ is divisible by $p$, then, however 
      large the numbers $a$, $b$, $c$, $d$ may be, it is always possible 
      to exhibit a similar form $x^{2}+y^{2}+z^{2}+v^{2}$ divisible by 
      $p$ in such a way that the single numbers $x$, $y$, $z$, $v$ are 
      no greater than half the number $p$.
\item If $p$ is a prime number and therefore odd, the single numbers $x$, 
      $y$, $z$, $v$ will be smaller than $\tfrac{1}{2}p$, so 
      $x^2+y^2+z^2+v^2 < 4 \cdot \frac14\,p^2 = p^2.$
\item If $p$ is any prime number, it will certainly be the sum of four or
      fewer squares.
\end{enumerate}
Euler remarks that (2) ``is the theorem on which the whole matter 
depends, and which I cannot yet prove''. The other claims are proved 
by him except for the fifth; here Euler writes
``The proof of this is particularly remarkable, but somewhat cumbersome; 
  if you like, it can make up the contents of an entire letter in the 
future''. A modern proof (actually it goes back to Minding \cite{Minding}) 
of a statement slightly weaker than 5 goes like this: the quadratic 
polynomials $-x^2$ and $1+y^2$ each attain $\frac{p+1}2$ distinct 
values modulo $p$, hence there must exist $x, y$ with 
$1+y^2 \equiv - x^2 \bmod p$, and then $p \mid x^2 + y^2 + 1$. 

The last claim is proved by descent: if there is a counterexample $p$, 
the previous propositions allow Euler to find a prime  $q < p$ which 
cannot be written as a sum of four squares: contradiction!

In [EG140; June 16, 1749], Goldbach takes up a special case of
Euler's missing lemma and writes
\begin{quote}
On the other hand, I think the proof of this proposition is within
my power: If any number is the sum of four odd squares, the same number is
also the sum of four even squares, or: four odd squares equal to $8m+4$
being given, there are also four squares for the number $2m+1$.
\end{quote}

In his reply [EG141; July 26, 1749], Euler proves this remark
as follows: 
\begin{quote}
Let $ 8m+4 = (2a+1)^2 + (2b+1)^2 + (2c+1)^2+( 2d+1)^2; $
then, on dividing by $2$, since 
$\frac{(2p+1)^2+(2q+1)^2}{2} = (p+q+1)^2+(p-q)^2$, 
$$ 4m+2 = (a+b+1)^2 + (a-b)^2 + (c+d+1)^2 + (c-d)^2, $$
so $4m+2=4\,\square $. Since, however, $4m+2$ is an oddly even number, 
two of these four squares must be even and two odd\footnote{If an odd
number of squares is odd, then the sum of squares is odd; thus there
must be $0$, $2$ or $4$ even squares. In the first and the third case,
the sum is divisible by $4$.}. 
So one will have
$$ 4m+2 = (2p+1)^2 + (2q+1)^2 + 4r^2 + 4s^2, $$
therefore
$$ 2m+1 = (p+q+1)^2+(p-q)^2+(r+s)^2+(r-s)^2, $$
and consequently 
$$ 8m+4 = 4(p+q+1)^2 + 4(p-q)^2 + 4(r+s)^2 + 4(r-s)^2, $$
QED.
\end{quote}

In slightly modernized form, we can formulate the essence of Euler's 
result as follows:

\begin{lem}
If $2n = a^2 + b^2 + c^2 + d^2$ is a sum of four squares, 
then so is $n$.
\end{lem}

\begin{proof}
We can permute $a$, $b$, $c$ and $d$ in such a way that 
$a - b$ and $c - d$ are even. But then 
$$ n = \Big(\frac{a+b}2\Big)^2 + \Big(\frac{a-b}2\Big)^2
       +  \Big(\frac{c+d}2\Big)^2 + \Big(\frac{c-d}2\Big)^2, $$
and we are done.
\end{proof}

Goldbach's remark and the simplicity of the proof lead Euler to
the realization that he could go further; in the same letter, 
Euler treats the analogous 

\begin{lem}
If $3n = F^2 + G^2 + H^2 + K^2$ is a sum of four squares, 
then so is $n$.
\end{lem}

\begin{proof}
We can write $F = f + 3r$, $G = g + 3s$, $H = h + 3t$ and $K = k + 3u$.
Up to permutation and choices of signs, there are the following cases:
\begin{enumerate}
\item $f = g = h = k = 0$. Then $n = 3(r^2+s^2+t^2+u^2)$, and the
      product formula yields the claim.
\item $f = g = h = 1$, $k = 0$. Then
      \begin{align*}
       n & = 1 + 2a + 2b + 2c + 3a^2 + 3b^2 + 3c^2 + 3d^2 \\
         & = (1+a+b+c)^2 + (a-b+d)^2 + (a-c-d)^2 + (b-c+d)^2.
      \end{align*}  
\end{enumerate}
This completes the proof.
\end{proof}

Euler treats the case $p = 5$ in a similar way, but gets stuck
with $p = 7$ (he does not see how to write the expression
$$ A = 2 + 2a + 4b + 6c + 7a^2 + 7b^2 + 7c^2 + 7d^2 $$ 
resulting from $(f,g,h,k) = (0,1,2,3)$ as a sum of four squares). 

Euler returns to the case $p = 7$ in the postscript of his letter:

\begin{quote}
PS. The theorem for $7A=\fbox{$4$}$, which I did not fully
execute, is completed by the following general theorem:
\end{quote}

\begin{thm}
Setting $m=a^2+b^2+c^2+d^2$, if $mA=\fbox{$4$}$,
then also $A=\fbox{$4$}$.
\end{thm}

\begin{proof} 
Let
$$ mA = (f+mp)^2 + (g+mq)^2 + (h+mr)^2 + (k+ms)^{2} $$
and 
\begin{equation}\label{Ecrux}
   f^2+g^2+h^2+k^2 = (a^2+b^2+c^2+d^2)(x^2+y^2+z^2+v^2); 
\end{equation}
then 
\begin{align*}
f & = ax+by+cz+dv  & g & = bx-ay-dz+cv \\
h & = cx+dy-az-bv  & k & = dx-cy+bz-av,
\end{align*}
and one gets 
$$ A = x^2+y^2+z^2+v^2+2(fp+gq+hr+ks) + m(p^2+q^2+r^2+s^2); $$
but from this one finds 
\begin{align*}
A & = (ap+bq+cr+ds+x)^2 + (aq-bp+cs-dr-y)^2 \\
  & \ \ + (ar-bs-cp+dq-z)^2 + (as+br-cq-dp-v)^2,
\end{align*}
so $A=\fbox{$4$}${} in whole numbers, QED.
\end{proof}

This looks exactly like the missing lemma in Euler's plan for proving
the Four-Squares Theorem. On the other hand, Euler later repeatedly
said that he did not have a proof of this lemma, and eventually
congratulated Lagrange on his proof of the theorem. So something
must be missing. In fact it is not clear where (\ref{Ecrux}) comes 
from. For small $m$, this identity can be checked by hand, which 
is what Euler did for $m = 2, 3, 5$ and $7$. What Euler failed to 
see at this point is that a rather simple induction proof now 
completes the proof of the Four-Squares Theorem.

\section{The Proof of the Four-Squares Theorem \`a la Euler}

In this section we will show that it is not difficult 
to complete the proof of the Four Squares Theorem by induction
using the formulas contained in Euler's letter n\higho\, 141. 
Instead of faithfully reproducing this proof here, we will use
linear algebra to abbreviate calculations. 
To this end, we consider the matrices
$$ \cM[a,b,c,d] = \left(\begin{matrix}
                    a &  b &  c &  d  \\
                   -b &  a &  d & -c  \\
                   -c & -d &  a &  b  \\
                   -d &  c & -b &  a \end{matrix} \right). $$

\begin{lem}
The product formula can be written in the form
$$ \cM[a,b,c,d]^* \cM[f,g,h,k] = \cM[r,s,t,u] $$
where  $A^*$ denotes the transpose of $A$, and where
\begin{align*}
      r & =  af + bg + ch + dk & s & = ag - bf + ck - dh \\
      t & =  ah - bk - cf + dg & t & = ak + bh - cg - df.
\end{align*}
In particular, $A^* A = m\cI$ for $A = \cM[a,b,c,d]$, where 
$m = a^2+b^2+c^2+d^2$ and $\cI$ is the $4 \times 4$-identity matrix.  
\end{lem}

We would like to prove the following theorem by induction on $m$:

\begin{thm}
Every positive integer $m$ is a sum of four squares. 

Moreover, if $mA = F^2 + G^2 + H^2 + K^2$ for integers $F, G, H, K$,
then there exist integers  $a, b, c, d$ and $x, y, z, v$ such that

\begin{equation}\label{ErelE}
    \begin{cases}
             m  = a^2 + b^2 + c^2 + d^2,  \quad A = x^2 + y^2 + z^2 + v^2,
       \quad \text{and} \\
  \cM[F,G,H,K] = \cM[a,b,c,d]^* \cM[x,y,z,v]. \end{cases}
\end{equation}
\end{thm}

\begin{proof}
The theorem holds for $m = 1$ and $a=1$, $b = c = d = 0$, $x = F$, \ldots,
$v = K$.

We will now prove the following steps:
\begin{enumerate}
\item[1.] $m$ is a sum of four squares.
\item[2.] (\ref{ErelE}) holds for all $A < m$: this follows from the induction
      assumption by switching the roles of $m$ and $A$.
\item[3.] (\ref{ErelE}) holds for all $A \ge m$: this is Euler's
      part of the proof.
\end{enumerate}
Ad 1. Assume that the Theorem holds for all natural numbers $< m$. 
   If $m$ is not squarefree, say $m = m_1n^2$ for $n > 1$, then $m_1$ 
   is a sum of four squares by induction assumption, hence so is $m$.

   If $m$ is squarefree, we solve the congruence $f^2 + g^2 \equiv -1 \bmod p$ 
   for every prime $p \mid m$ and use the Chinese Remainder Theorem to find 
   integers $A, F, G$ such that $mA = F^2 + G^2 + 1$. Reducing $F$ and $G$ 
   modulo $m$ in such a way that the squares of the remainders are minimal 
   shows that we may assume that $A < m$. The induction assumption (we have 
   to switch the roles of $m$ and $A$) shos that (\ref{ErelE}) holds.

Ad 3. Write $mA = F^2 + G^2 + H^2 + K^2$, and define integers
$-\frac m2 < f, g, h, k \le \frac m2$ using the Euclidean algorithm: 
$F = f + mr$, $G = g + ms$, $H = h + mt$ and $K = k + mu$. Then we have 
$\cM[F,G,H,K] = \cM[f,g,h,k] + m \cM[r,s,t,u]$. Now
$f^2 + g^2 + h^2 + k^2 \le m^2$ is divisible by $m$, say $= mB$
for some number $B \le m$. If $B = m$, then $f^2 = g^2 = h^2 = k^2$
and our claim holds; if $B < m$, then the induction assumption 
guarantees the existence of integers $x, y, z, v$ with 
$\cM[f,g,h,k] = \cM[a,b,c,d]^*\cM[x,y,z,v]$. Using
$m\cI = \cM[a,b,c,d,]^*\cM[a,b,c,d,]$ we now find
\begin{align*}
  \cM[F,G,H,K] & = \cM[f,g,h,k] + m\cM[r,s,t,u] \\
            & = \cM[a,b,c,d]^*\cM[x,y,z,v]
                + \cM[a,b,c,d,]^*\cM[x,y,z,v] \cM[r,s,t,u] \\
            & = \cM[a,b,c,d]^*(\cM[x,y,z,v] + \cM[a,b,c,d,] \cM[r,s,t,u]) \\
            & = \cM[a,b,c,d]^* \cM[X,Y,Z,V] 
\end{align*}
with $\cM[X,Y,Z,V] = \cM[x,y,z,v] + \cM[r,s,t,u]\cM[a,b,c,d,]$, i.e.
\begin{align*}
   X & = x + ar + bs + ct + du, & 
   Y & = y - as + br - cu + dt, \\
   Z & = z - at + bu + cr - ds, & 
   V & = v - au - bt + cs + dr.
\end{align*}
\begin{align*}
 \text{From } \qquad 
 mA\cI & = \cM[F,G,H,K]^*\cM[F,G,H,K] \\
       & = \cM[X,Y,Z,V]^*\cM[a,b,c,d] \cM[a,b,c,d]^* \cM[X,Y,Z,V]   \\
       & = m(X^2+Y^2+Z^2+V^2 
\end{align*}
we deduce that $X^2+Y^2+Z^2+V^2 = A$.  
\end{proof}

\medskip\noindent{\bf Remark.}
The matrices $\cM[r,s,t,u]$ form a ring isomorphic to the Lipschitz 
quaternions. The proof of the Four-Squares Theorem due to Lagrange 
and Euler was first translated into the language of quaternions by 
Hurwitz \cite{Hurw}.

\section*{Acknowledgement}
I thank Martin Mattm\"uller for the crucial observation
that Euler's postscript might be sufficient for proving the
Four-Squares Theorem, as well as for his help in translating 
from Latin. I also thank Norbert Schappacher for his valuable 
comments.

\end{document}